\documentclass{article}
\usepackage[left=3cm,right=3cm,top=3cm,bottom=2cm]{geometry}
\usepackage[colorlinks]{hyperref}
\hypersetup{linkcolor=blue, citecolor=blue}
\usepackage{amsmath, enumitem, amssymb, amsthm, bm, mathrsfs, braket, graphicx, tikz, comment, mathtools, cleveref}

\title{log-Coulomb gases in the projective line of a $p$-field}
\author{Joe Webster}
\date{\today}

\newcommand\R{\mathbb{R}}

\newcommand\N{\mathbb{N}}
\newcommand\Z{\mathbb{Z}}
\newcommand\C{\mathbb{C}}

\newcommand\F{\mathbb{F}}
\renewcommand\P{\mathbb{P}}
\newcommand\spl{{\bm{\pitchfork}}}
\newcommand\ptn{{\pitchfork}}
\DeclareMathOperator\re{Re}

\theoremstyle{definition}
\newtheorem{theorem}{Theorem}[section]
\newtheorem{proposition}[theorem]{Proposition}
\newtheorem{lemma}[theorem]{Lemma}
\newtheorem{definition}[theorem]{Definition}

\newtheorem{example}[theorem]{Example}
\newtheorem{notation}[theorem]{Notation}

\numberwithin{equation}{subsection}

\begin{document}

\maketitle

\begin{abstract}
This article extends recent results on log-Coulomb gases in a $p$-field $K$ (i.e., a nonarchimedean local field) to those in its projective line $\P^1(K)$, where the latter is endowed with the $PGL_2$-invariant Borel probability measure and spherical metric. Our first main result is an explicit combinatorial formula for the canonical partition function of log-Coulomb gases in $\P^1(K)$ with arbitrary charge values. Our second main result is called the ``$(q+1)$th Power Law", which relates the grand canonical partition functions for one-component gases in $\P^1(K)$ (where all particles have charge 1) to those in the open and closed unit balls of $K$ in a simple way. The final result is a quadratic recurrence for the canonical partition functions for one-component gases in both unit balls of $K$ and in $\P^1(K)$. In addition to efficient computation of the canonical partition functions, the recurrence provides their ``$q\to 1$" limits and ``$q\mapsto q^{-1}$" functional equations.
\end{abstract}

\let\thefootnote\relax

\tableofcontents

\newpage
\section{Introduction}
\subsection{Canonical partition functions for log-Coulomb gases}\label{1_1}
Let $X$ be a topological space endowed with a metric $d$ and a finite positive Borel measure $\mu$ satisfying $\mu^N\{(x_1,\dots,x_N)\in X^N:x_i=x_j\text{ for some $i\neq j$}\}=0$ for every $N\geq 1$. A \emph{log-Coulomb gas} in $X$ is a statistical model described as follows: Consider $N\geq 1$ particles with fixed charge values $\mathfrak{q}_1,\dots,\mathfrak{q}_N\in\R$ and corresponding variable locations $x_1,\dots,x_N\in X$. Whether the charge values are distinct or not, we assume particles are distinguished by the labels $1,2,\dots,N$, so that unique configurations of the system correspond to unique tuples $(x_1,\dots,x_N)\in X^N$. Each tuple is called a \emph{microstate} and has an \emph{energy} defined by $$E(x_1,\dots,x_N):=\begin{cases}-\sum_{1\leq i<j\leq N}\mathfrak{q}_i\mathfrak{q_j}\log d(x_i,x_j)&\text{if $x_i\neq x_j$ for all $i<j$},\\\infty&\text{otherwise}.\end{cases}$$ Note that $E^{-1}(\infty)$ has measure 0 in $X^N$ by our choice of $\mu$, and that $E$ is identically zero if $N=1$. We assume the system is in thermal equilibrium with a heat reservoir at some \emph{inverse temperature} $\beta>0$, so that its microstates form a canonical ensemble distributed according to the density $e^{-\beta E(x_1,\dots,x_N)}$. The \emph{canonical partition function} $\beta\mapsto Z_N(X,\beta)$ is defined as the total mass of this density, namely
\begin{equation}\label{canonical}
Z_N(X,\beta):=\int_{X^N}e^{-\beta E(x_1,\dots,x_N)}\,d\mu^N=\int_{X^N}\prod_{1\leq i<j\leq N}d(x_i,x_j)^{\mathfrak{q}_i\mathfrak{q}_j\beta}\,d\mu^N~.
\end{equation}
Given $(X,d,\mu)$ and $\mathfrak{q}_1,\dots,\mathfrak{q}_N$, the explicit formula for $Z_N(X,\beta)$ is of primary interest, as it yields fundamental relationships between the observable parameters of the system and its temperature. For instance, the system's dimensionless free energy, mean energy, and energy fluctuation (variance) are respectively given by $-\log Z_N(X,\beta)$, $-\partial/\partial\beta\log Z_N(X,\beta)$, and $\partial^2/\partial\beta^2\log Z_N(X,\beta)$, all of which are functions of $\beta$ (and hence of temperature). Below are two examples in which the formula for $Z_N(X,\beta)$ is known.

\begin{example}\label{arch}
If $X=\R$ with the standard metric $d$, the standard Gaussian measure $\mu$, and charge values $\mathfrak{q}_1=\dots=\mathfrak{q}_N=1$, then \emph{Mehta's integral formula} \cite{FW} states that $Z_N(\R,\beta)$ converges absolutely for all complex $\beta$ with $\re(\beta)>-2/N$, and in this case $$Z_N(\R,\beta)=\prod_{j=1}^N\frac{\Gamma(1+j\beta/2)}{\Gamma(1+\beta/2)}~.$$ At the special values $\beta\in\{1,2,4\}$, the probability density $\frac{1}{Z_N(\R,\beta)}e^{-\beta E(x_1,\dots,x_N)}$ coincides with the joint density of the $N$ eigenvalues $x_1,\dots,x_N$ (counted with multiplicity) of the $N\times N$ Gaussian orthogonal $(\beta=1)$, unitary $(\beta=2)$, and symplectic $(\beta=4)$ matrix ensembles.
\end{example}

\begin{example}\label{Z_3}
If $X=\Z_p$ with $d(x,y)=|x-y|_p$, the Haar probability measure $\mu$, and $N=3$ charges with $\mathfrak{q}_1=1$, $\mathfrak{q}_2=2$, and $\mathfrak{q}_3=3$, then a general theorem in \cite{Web20} implies that $Z_3(\Z_p,\beta)$ converges absolutely for all complex $\beta$ with $\re(\beta)>-1/6$, and in this case $$Z_3(\Z_p,\beta)=\frac{p^{11\beta}}{p^{2+11\beta}-1}\cdot\left((p-1)(p-2)+(p-1)^2\left[\frac{1}{p^{1+2\beta}-1}+\frac{1}{p^{1+3\beta}-1}+\frac{1}{p^{1+6\beta}-1}\right]\right)~.$$
\end{example}

Note that $\beta\mapsto Z_N(X,\beta)$ always extends to a complex domain containing the line $\re(\beta)=0$. To simultaneously treat all possible choices of $\mathfrak{q}_i$, we extend further to a subset of $\C^{N(N-1)/2}$ as follows: 

\begin{definition}\label{gen_Z}
For any $N\geq 0$ and $(X,d,\mu)$ as above, we write $\bm{s}$ for a complex tuple $(s_{ij})_{1\leq i<j\leq N}$ (the empty tuple if $N=0$ or $N=1$) and define $\mathcal{Z}_0(X,\bm{s}):=1$ and $$\mathcal{Z}_N(X,\bm{s}):=\int_{X^N}\prod_{1\leq i<j\leq N}d(x_i,x_j)^{s_{ij}}\,d\mu^N\qquad\text{for $N\geq 1$.}$$
\end{definition}

Once the formula and domain for $\bm{s}\mapsto\mathcal{Z}_N(X,\bm{s})$ are known, then for any choice of $\mathfrak{q}_1,\dots,\mathfrak{q}_N\in\R$, the formula and domain for $\beta\mapsto Z_N(X,\beta)$ follow by specializing $\mathcal{Z}_N(X,\bm{s})$ to ${s_{ij}=\mathfrak{q}_i\mathfrak{q}_j\beta}$. Thus, given $(X,d,\mu)$, the main problem is to determine the formula and domain of $\bm{s}\mapsto\mathcal{Z}_N(X,\bm{s})$.

\subsection{$p$-fields, projective lines, and splitting chains}\label{1_2}
Of our two main goals in this paper, the first is to determine the explicit formula and domain of $\bm{s}\mapsto\mathcal{Z}_N(\P^1(K),\bm{s})$, where $K$ is a $p$-field and its projective line $\P^1(K)$ is endowed with a natural metric and measure. To make this precise, we briefly recall well-known properties of $p$-fields (see \cite{Weil}, for instance) and establish some notation. Fix a $p$-field $K$, write $|\cdot|$ for its canonical absolute value, write $d$ for the associated metric (i.e., $d(x,y):=|x-y|$), and define $$R:=\{x\in K:|x|\leq 1\}\qquad\text{and}\qquad P:=\{x\in K:|x|<1\}~.$$ The closed unit ball $R$ is the maximal compact subring of $K$, the open unit ball $P$ is the unique maximal ideal in $R$, and the group of units is $R^\times=R\setminus P=\{x\in K:|x|=1\}$. The residue field $R/P$ is isomorphic to $\F_q$ for some prime power $q$, and there is a canonical isomorphism of the cyclic group $(R/P)^\times$ onto the group of $(q-1)$th roots of unity in $K$. Fixing a primitive such root $\xi\in K$ and sending $P\mapsto 0$ extends the isomorphism $(R/P)^\times\to\{1,\xi,\dots,\xi^{q-2}\}$ to a bijection $R/P\to\{0,1,\xi,\dots,\xi^{q-2}\}$ with inverse $y\mapsto y+P$. Therefore $\{0,1,\xi,\dots,\xi^{q-2}\}$ is a complete set of representatives for the cosets of $P\subset R$, i.e.,
\begin{equation}\label{R_decomp}
R=P\sqcup\underbrace{(1+P)\sqcup(\xi+P)\sqcup\dots\sqcup(\xi^{q-2}+P)}_{R^\times}~.
\end{equation}
Fix a uniformizer $\pi\in K$ (any element satisfying $P=\pi R$) and let $\mu$ be the unique additive Haar measure on $K$ satisfying $\mu(R)=1$. The open balls in $K$ are precisely the sets of the form $y+\pi^vR$ with $y\in K$ and $v\in\Z$, and every such ball is compact with measure equal to its radius, i.e., $\mu(y+\pi^vR)=|\pi^v|=q^{-v}$. In particular, each of the $q$ cosets of $P$ in \eqref{R_decomp} is a compact open ball with measure (and radius) $q^{-1}$, and two elements $x,y\in R$ satisfy $|x-y|=1$ if and only if they belong to different cosets. Henceforth, we reserve the symbols $K$, $|\cdot|$, $d$, $R$, $P$, $q$, $\xi$, $\pi$, and $\mu$ for the items above, and we distinguish $|\cdot|$ from the standard absolute value on $\C$ by writing $|\cdot|_\C$ for the latter.\\

We now recall some useful facts from \cite{FiliPetsche} in our present notation. The projective line of $K$ is the quotient space $\P^1(K)=(K^2\setminus\{(0,0)\})/\sim$, where $(x_0,x_1)\sim(y_0,y_1)$ if and only if $y_0=\lambda x_0$ and $y_1=\lambda x_1$ for some $\lambda\in K^\times$. Thus we may understand $\P^1(K)$ concretely as the set of symbols $[x_0:x_1]$ with $(x_0,x_1)\in K^2\setminus\{(0,0)\}$, subject to the relation $[\lambda x_0:\lambda x_1]=[x_0:x_1]$ for all $\lambda\in K^\times$ and endowed with the topology induced by the quotient $(x_0,x_1)\mapsto[x_0:x_1]$. The projective line is compact and metrizable by the \emph{spherical metric} $\delta:\P^1(K)\times\P^1(K)\to\{0\}\cup\{q^{-v}:v\in\Z_{\geq 0}\}$, which is defined via
\begin{equation}\label{sph_met_def}
\delta([x_0:x_1],[y_0:y_1]):=\frac{|x_0y_1-x_1y_0|}{\max\{|x_0|,|x_1|\}\cdot\max\{|y_0|,|y_1|\}}~.
\end{equation}
In particular, every open set in $\P^1(K)$ is a union of balls of the form
\begin{equation}\label{ball_def}
B_v[x_0:x_1]:=\{[y_0:y_1]\in\P^1(K):\delta([x_0:x_1],[y_0:y_1])\leq q^{-v}\}
\end{equation}
with $[x_0:x_1]\in\P^1(K)$ and $v\in\Z_{\geq0}$, and every such ball is open and compact. The projective linear group $PGL_2(R)$ is the quotient of $GL_2(R)=\{A\in M_2(R):\det(A)\in R^\times\}$ by its center, namely $Z=\{\left(\begin{smallmatrix}\lambda&0\\0&\lambda\end{smallmatrix}\right):\lambda\in R^\times\}\cong R^\times$. It is straightforward to verify that the rule $$\phi[x_0:x_1]:=[ax_0+bx_1:cx_0+dx_1]~,$$ where $\phi\in PGL_2(R)$ and $\left(\begin{smallmatrix}a&b\\c&d\end{smallmatrix}\right)\in GL_2(R)$ is any representative of $\phi$, gives a well-defined transitive action of $PGL_2(R)$ on $\P^1(K)$.  

\begin{lemma}[$PGL_2(R)$-invariance \cite{FiliPetsche}]\label{invariance} The spherical metric satisfies $$\delta(\phi[x_0:x_1],\phi[y_0:y_1])=\delta([x_0:x_1],[y_0:y_1])$$ for all $\phi\in PGL_2(R)$ and all $[x_0:x_1],[y_0:y_1]\in\P^1(K)$. There is also a unique Borel probability measure $\nu$ on $\P^1(K)$ satisfying $$\nu(\phi(M))=\nu(M)$$ for all $\phi\in PGL_2(R)$ and all Borel subsets $M\subset\P^1(K)$. In particular, for each $v\in\Z_{\geq0}$ the relation $\phi(B_v[x_0:x_1])=B_v(\phi[x_0:x_1])$ defines a transitive $PGL_2(R)$ action on the set of balls of radius $q^{-v}$, and thus $\nu(B_v[x_0:x_1])$ depends only on $v$. 
\end{lemma}

It is routine to verify that $\mu^N\{(x_1,\dots,x_N)\in(y+\pi^vR)^N:x_i=x_j\text{ for some $i\neq j$}\}=0$ and $$\nu^N\{([x_{1,0}:x_{1,1}],\dots,[x_{N,0}:x_{N,1}])\in(\P^1(K))^N:[x_{i,0}:x_{i,1}]=[x_{j,0}:x_{j,1}]\text{ for some $i\neq j$}\}=0$$ for all $N\geq 1$, so we have suitable metrics and measures to define log-Coulomb gases in $X=y+\pi^vR$ and $X=\P^1(K)$. \Cref{gen_Z} specializes to these $X$ as follows:

\begin{definition}\label{main_Z_def}
With $N\geq 0$ and $\bm{s}$ as before, we have $\mathcal{Z}_0(y+\pi^vR,\bm{s})=\mathcal{Z}_0(\P^1(K),\bm{s})=1$, and
\begin{align*}
\mathcal{Z}_N(y+\pi^vR,\bm{s})&=\int_{(y+\pi^vR)^N}\prod_{1\leq i<j\leq N}|x_i-x_j|^{s_{ij}}\,d\mu^N\qquad\text{and}\\
\mathcal{Z}_N(\P^1(K),\bm{s})&=\int_{(\P^1(K))^N}\prod_{1\leq i<j\leq N}\delta([x_{i,0}:x_{i,1}],[x_{j,0}:x_{j,1}])^{s_{ij}}\,d\nu^N
\end{align*}
for $N\geq 1$. The first integral is independent of $y$, and equal to $\mathcal{Z}_N(R,\bm{s})$ if $v=0$ or $\mathcal{Z}_N(P,\bm{s})$ if $v=1$.
\end{definition}
The formulas and domains of absolute convergence for the integrals above can be stated neatly in terms of the following items from \cite{Web20}:

\begin{definition}[Splitting chains]\label{splchdef}
A \emph{splitting chain} of order $N\geq 2$ and length $L\geq 1$ is a tuple $\spl=(\ptn_0,\dots,\ptn_L)$ of partitions of $[N]=\{1,\dots,N\}$ satisfying $$\{[N]\}=\ptn_0>\ptn_1>\ptn_2>\dots>\ptn_L=\{\{1\},\dots,\{N\}\}~.$$
\begin{itemize}
\item[(a)] Each non-singleton part $\lambda\in\ptn_0\cup\ptn_1\cup\dots\cup\ptn_L$ is called a \emph{branch} of $\spl$. We write $\mathcal{B}(\spl)$ for the set of all branches of $\spl$, i.e., $$\mathcal{B}(\spl):=(\ptn_0\cup\dots\cup\ptn_{L-1})\setminus\ptn_L~.$$
\item[(b)] Each $\lambda\in\mathcal{B}(\spl)$ must appear in a final partition $\ptn_\ell$ before refining into two or more parts in $\ptn_{\ell+1}$, so we define its \emph{depth} $\ell_\spl(\lambda)\in\{0,1,\dots,L-1\}$ and \emph{degree} $\deg_\spl(\lambda)\in\{2,3,\dots,N\}$ by $$\ell_\spl(\lambda):=\max\{\ell:\lambda\in\ptn_\ell\}\qquad\text{and}\qquad\deg_\spl(\lambda):=\#\{\lambda'\in\ptn_{\ell_\spl(\lambda)+1}:\lambda'\subset\lambda\}~.$$
\item[(c)] We say that $\spl$ is \emph{reduced} if each $\lambda\in\mathcal{B}(\spl)$ satisfies $\lambda\in\ptn_\ell\iff\ell=\ell_\spl(\lambda)$.
\end{itemize}
Write $\mathcal{R}_N$ for the set of reduced splitting chains of order $N$ and define $$\Omega_N:=\bigcap_{\substack{\lambda\subset[N]\\\#\lambda>1}}\left\{\bm{s}:\re(e_\lambda(\bm{s}))>0\right\}\qquad\text{where}\qquad e_\lambda(\bm{s}):=\#\lambda-1+\sum_{\substack{i<j\\i,j\in\lambda}}s_{ij}~.$$ 
\end{definition}

Proposition 3.15 and Theorem 2.6(c) in \cite{Web20} imply the following proposition, which shall be generalized slightly in order to prove the main results of this paper:

\begin{proposition}\label{R_prop}
For $N\geq 2$, the integral $\mathcal{Z}_N(R,\bm{s})$ converges absolutely if and only if $\bm{s}\in\Omega_N$, and in this case it can be written as the finite sum $$\mathcal{Z}_N(R,\bm{s})=q^{\sum_{i<j}s_{ij}}\sum_{\spl\in\mathcal{R}_N}\prod_{\lambda\in\mathcal{B}(\spl)}\frac{(q-1)_{\deg_{\spl}(\lambda)-1}}{q^{e_\lambda(\bm{s})}-1}~.$$ Here $(q-1)_n$ stands for the degree $n$ falling factorial $(z)_n=z(z-1)\dots(z-n+1)\in\Z[z]$ evaluated at the integer $z=q-1$.
\end{proposition}

\section{Statement of results}\label{2}
\subsection{The projective analogue}\label{2_1}

Our first main result is the following analogue of \Cref{R_prop}:

\begin{theorem}\label{main1}
For $N\geq 2$, the integral $\mathcal{Z}_N(\P^1(K),\bm{s})$ converges absolutely if and only if $\bm{s}\in\Omega_N$, and in this case it can be written as the finite sum $$\mathcal{Z}_N(\P^1(K),\bm{s})=\frac{1}{(q+1)^{N-1}}\sum_{\spl\in\mathcal{R}_N}\frac{q^{N+\sum_{i<j}s_{ij}}+1-\deg_{\spl}([N])}{q+1-\deg_{\spl}([N])}\prod_{\lambda\in\mathcal{B}(\spl)}\frac{(q-1)_{\deg_{\spl}(\lambda)-1}}{q^{e_\lambda(\bm{s})}-1}~.$$ The summand for each $\spl\in\mathcal{R}_N$ is defined for all prime powers $q$, as the denominator $q+1-\deg_\spl([N])$ is cancelled by the factor $(q-1)_{\deg_\spl([N])-1}$ inside the product over $\lambda\in\mathcal{B}(\spl)$.
\end{theorem}

The evident similarities between \Cref{R_prop} and \Cref{main1} follow from explicit relationship between the metrics and measures on $K$ and those on $\P^1(K)$. These relationships also play a role in the upcoming results, so they are worth recalling now. Note that $[x_0:x_1]\neq[1:0]$ if and only if $x_1\neq 0$, in which case $x=x_0/x_1$ is the unique element of $K$ satisfying $[x:1]=[x_0:x_1]$. Therefore the rule $\iota(x):=[x:1]$ defines a bijection $\iota:K\to\P^1(K)\setminus\{[1:0]\}$ and relates the metric structures of $K$ and $\P^1(K)$ in a simple way: Given $x,y\in K$, \eqref{sph_met_def} implies $\delta(\iota(x),[1:0])=(\max\{1,|x|\})^{-1}$ and
\begin{equation}\label{metric_rule}
\delta(\iota(x),\iota(y))=\begin{cases}|x-y|&\text{if $x,y\in R$},\\1&\text{if $x\in R$ and $y\notin R$},\\|1/x-1/y|&\text{if }x,y\notin R.\end{cases}
\end{equation}
Using the definitions \eqref{ball_def} and $v_K(x):=-\log_q|x|$ for $x\in K^\times$, along with \eqref{metric_rule} and the \emph{strong triangle equality} (i.e., $|x+y|=\max\{|x|,|y|\}$ whenever $|x|\neq|y|$), one easily verifies that
\begin{equation}\label{ball_rule}
\iota(y+\pi^vR)=\begin{cases}B_v(\iota(y))&\text{if }y\in R,\\B_{v-2v_K(y)}(\iota(y))&\text{if }y\notin R,\end{cases}
\end{equation}
whenever $y\in K$ and $v\in\Z_{>0}$. That is, $\iota$ sends the open ball of radius $r\in(0,1)$ centered at $y\in K$ onto the open ball of radius $r/\max\{1,|y|^2\}$ centered at $\iota(y)\in\P^1(K)\setminus\{[1:0]\}$, so $\iota:K\to\P^1(K)\setminus\{[1:0]\}$ is a homeomorphism that restricts to an isometry on $R$ and a contraction on $K\setminus R$.

The map $\iota$ also relates the measures on $K$ and $\P^1(K)$ in a simple way: Given $v>0$ and a complete set of representatives $y_1,\dots,y_{q^v}\in R$ for the cosets of $\pi^vR\subset R$, applying \eqref{ball_rule} to the partition $R=(y_1+\pi^vR)\sqcup\dots\sqcup(y_{q^v}+\pi^vR)$ yields $$\iota(R)=B_v[y_1:1]\sqcup\cdots\sqcup B_v[y_{q^v}:1]~.$$ Therefore $PGL_2(R)$-invariance of $\nu$ (\Cref{invariance}) implies $\nu(\iota(R))=q^v\cdot\nu(B_v[0:1])$. On the other hand, $$\iota(K\setminus R)=\iota(\{x:|x|\geq q\})=\{\iota(x):\delta(\iota(x),[1:0])\leq q^{-1}\}=B_1[1:0]\setminus\{[1:0]\}$$ implies $\iota(R)\sqcup B_1[1:0]=\iota(R)\sqcup\iota(K\setminus R)\sqcup\{[1:0]\}=\P^1(K)$, which has measure 1. But $\nu(\iota(R))=q\cdot\nu(B_1[1:0])$, so the measure of $\iota(R)$ must be $q/(q+1)$, and therefore every ball $B_v[x_0:x_1]\subset\P^1(K)$ with $v>0$ has measure $q^{-v}\cdot q/(q+1)$. Combining this with \eqref{ball_rule}, one concludes that the measure $\nu$ on $\P^1(K)\setminus\{[1:0]\}$ pulls back along $\iota$ to an explicit measure on $K$:
\begin{equation}\label{measure_rule}
\nu(\iota(M))=\frac{q}{q+1}\int_M\left(\max\{1,|x|^2\}\right)^{-1}\,d\mu\qquad\text{for any Borel subset }M\subset K~.
\end{equation}

Finally, \eqref{R_decomp} and \eqref{ball_rule} give a nice refinement of $\P^1(K)=\iota(R)\sqcup B_1[1:0]$ in terms of the $(q-1)$th roots of unity in $K$, which should be understood as the projective analogue of \eqref{R_decomp}:
\begin{equation}\label{P(K)_decomp}
\P^1(K)=B_1[0:1]\sqcup\underbrace{B_1[1:1]\sqcup B_1[\xi:1]\sqcup\dots\sqcup B_1[\xi^{q-2}:1]}_{\iota(R^\times)}\sqcup B_1[1:0]~.
\end{equation}
Indeed, all $q+1$ of the parts in the partition are balls with measure $1/(q+1)$ and radius $q^{-1}$, and two points $[x_0:x_1],[y_0:y_1]\in\P^1(K)$ satisfy $\delta([y_0:y_1],[y_0:y_1])=1$ if and only if $[x_0:x_1]$ and $[y_0:y_1]$ belong to different parts. Note that $\iota$ sends $R^\times$ onto the ``equator" $\iota(R^\times)$, i.e., the set of points in $\P^1(K)$ with $\delta$-distance 1 from both the ``south pole" $[0:1]$ and the ``north pole" $[1:0]$.

\subsection{Relationships between grand canonical partition functions}\label{2_2}
So far we have only considered log-Coulomb gases with $N$ labeled (and hence distinguishable) particles. Our second main result concerns the situation in which all particles are identical with charge $\mathfrak{q}_i=1$ for all $i$, in which case a microstate $(x_1,\dots,x_N)\in X^N$ is regarded as unique only up to permutations of its entries. Since the energy $E(x_1,\dots,x_N)$ and measure on $X^N$ are invariant under such permutations, each unlabeled microstate makes the contribution $e^{-\beta E(x_1,\dots,x_N)}d\mu^N$ to the integral $Z_N(X,\beta)$ in \eqref{canonical} precisely $N!$ times. Therefore the canonical partition function for the unlabeled microstates is given by $Z_N(X,\beta)/N!$. We further assume that the system exchanges particles with the heat reservoir with chemical potential $\eta$ and define the \emph{fugacity parameter} $f=e^{\eta\beta}$. In this situation the particle number $N\geq 0$ is treated as a random variable and the canonical partition function is replaced by the \emph{grand canonical partition function}
\begin{equation}\label{granddef}
Z(f,X,\beta):=\sum_{N=0}^\infty Z_N(X,\beta)\frac{f^N}{N!}
\end{equation}
with the familiar convention $Z_0(X,\beta)=1$. Many properties of the system can be deduced from the grand canonical partition function. For instance, if $\beta>0$ is fixed and $Z_N(X,\beta)$ is sub-exponential in $N$, then $Z(f,X,\beta)$ is analytic in $f$ and the expected number of particles in the system is given by $f\frac{\partial}{\partial f}\log(Z(f,X,\beta))$. The canonical partition function for each $N\geq 0$ can also be recovered by evaluating the $N$th derivative of $Z(f,X,\beta)$ with respect to $f$ at $f=0$.

We are interested in the examples $Z(f,R,\beta)$, $Z(f,P,\beta)$, and $Z(f,\P^1(K),\beta)$, which turn out to share several common properties and simple relationships. By setting $s_{ij}=\beta$ in \Cref{main_Z_def}, one sees that $|Z_N(R,\beta)|_\C$, $|Z_N(P,\beta)|_\C$, and $|Z_N(\P^1(K),\beta)|_\C$ are bounded above by 1 for all $N\geq 0$ and all $\beta>0$, and hence $Z(f,R,\beta)$, $Z(f,P,\beta)$, and $Z(f,\P^1(K),\beta)$ are analytic in $f$ when $\beta>0$. Sinclair recently found an elegant relationship between the first two, which is closely related to the partition of $R$ in \eqref{R_decomp}:

\begin{proposition}[The $q$th Power Law \cite{ChrisA}]\label{Chris}
For $\beta>0$ we have $$Z(f,R,\beta)=(Z(f,P,\beta))^q~.$$ 
\end{proposition}

Roughly speaking, the $q$th Power Law states that a log-Coulomb gas in $R$ exchanging energy and particles with a heat reservoir ``factors" into $q$ identical sub-gases (one in each coset of $P$) that exchange energy and particles with the reservoir. For $\beta>0$, note that the series equation $Z(f,R,\beta)=(Z(f,P,\beta))^q$ is equivalent to the coefficient identities
\begin{equation}\label{Rcoeff}
\frac{Z_N(R,\beta)}{N!}=\sum_{\substack{N_0+\dots+N_{q-1}=N\\N_0,\dots,N_{q-1}\geq 0}}\prod_{k=0}^{q-1}\frac{Z_{N_k}(P,\beta)}{N_k!}\qquad\text{for all $N\geq 0$.}
\end{equation}
The $\beta=1$ case of \eqref{Rcoeff} is given in \cite{BGMR}, in which the positive number $Z_N(R,1)/N!$ is recognized as the probability that a random monic polynomial in $R[x]$ splits completely in $R$. The more general $\beta>0$ case given in \cite{ChrisA} makes explicit use of the partition of $R$ into cosets of $P$ (as in \eqref{R_decomp}). In \Cref{3}, we will use the analogous partition of $\P^1(K)$ into $q+1$ balls (recall \eqref{P(K)_decomp}) to show that
\begin{equation}\label{P(K)coeff}
\frac{Z_N(\P^1(K),\beta)}{N!}=\sum_{\substack{N_0+\dots+N_q=N\\N_0,\dots,N_q\geq 0}}\prod_{k=0}^q\left(\frac{q}{q+1}\right)^{N_k}\frac{Z_{N_k}(P,\beta)}{N_k!}\qquad\text{for all $\beta>0$ and $N\geq 0$,}
\end{equation}
which immediately implies our second main result:
\begin{theorem}[The $(q+1)$th Power Law]\label{main2}
For all $\beta>0$ we have $$Z(f,\P^1(K),\beta)=(Z(\tfrac{qf}{q+1},P,\beta))^{q+1}~.$$
\end{theorem}

Like the $q$th Power Law, the $(q+1)$th Power Law roughly states that a log-Coulomb gas in $\P^1(K)$ exchanging energy and particles with a heat reservoir ``factors" into $q+1$ identical sub-gases in the balls $B_1[0:1]$, $B[1:1]$, $B[\xi:1]$, $\dots$, $B[\xi^{q-2}:1]$, $B_1[1:0]$ (each isometrically homeomorphic to $P$), with fugacity $\frac{qf}{q+1}$. The $q$th Power Law allows the $(q+1)$th Power Law to be written more crudely as
\begin{equation}\label{P(K)=RP}
Z(f,\P^1(K),\beta)=Z(\tfrac{qf}{q+1},R,\beta)\cdot Z(\tfrac{qf}{q+1},P,\beta)~,
\end{equation}
which is to say that the gas in $\P^1(K)$ ``factors" into two sub-gases: one in $\iota(R)$ and one in $B[1:0]$ (which are respectively isometrically homeomorphic to $R$ and $P$), both with fugacity $\frac{qf}{q+1}$.

\subsection{Functional equations and a quadratic recurrence}\label{2_3}
Although \Cref{R_prop} and \Cref{main1} provide explicit formulas for $Z_N(R,\beta)$ and $Z_N(\P^1(K),\beta)$, they are not efficient for computation because they require a complete list of reduced splitting chains of order $N$. For a practical alternative, we take advantage of both Power Laws and the following ideas from \cite{BGMR} and \cite{ChrisA}: Apply $Z(f,P,\beta)\cdot\frac{\partial}{\partial f}$ to the equation $Z(f,R,\beta)=(Z(f,P,\beta))^q$ to get $$Z(f,P,\beta)\cdot\frac{\partial}{\partial f}Z(f,R,\beta)=q\cdot Z(f,R,\beta)\cdot\frac{\partial}{\partial f}Z(f,P,\beta)~,$$ then expand both sides as power series in $f$ to obtain the coefficient equations
\begin{equation}\label{coeff}
\sum_{k=1}^N\frac{Z_{N-k}(P,\beta)}{(N-k)!}\frac{Z_k(R,\beta)}{(k-1)!}=q\cdot\sum_{k=1}^N\frac{Z_{N-k}(R,\beta)}{(N-k)!}\frac{Z_k(P,\beta)}{(k-1)!}\qquad\text{for all }N\geq 1.
\end{equation}
The identities $Z_j(P,\beta)=q^{-j-\binom{j}{2}\beta}Z_j(R,\beta)$ follow easily from \Cref{main_Z_def} and eliminate all instances of $Z_j(P,\beta)$ in \eqref{coeff} while introducing powers of the form $q^{-j-\binom{j}{2}\beta}$. For $N\geq 2$, a careful rearrangement of these powers, the factorials, and the terms in \eqref{coeff} yields the explicit recurrence $$\frac{Z_N(R,\beta)}{N!q^{\frac{1}{2}\binom{N}{2}\beta}}=\sum_{k=1}^{N-1}\frac{k}{N}\cdot\frac{\sinh\left(\frac{\log(q)}{2}\left[\left(N+\binom{N}{2}\beta\right)\left(1-\frac{2k}{N}\right)+1\right]\right)}{\sinh\left(\frac{\log(q)}{2}\left[\left(N+\binom{N}{2}\beta\right)-1\right]\right)}\cdot\frac{Z_{N-k}(R,\beta)}{(N-k)!q^{\frac{1}{2}\binom{N-k}{2}\beta}}\cdot\frac{Z_k(R,\beta)}{k!q^{\frac{1}{2}\binom{k}{2}\beta}}~.$$ The expression at left is identically 1 if $N=0$ or $N=1$, so induction confirms that it is polynomial in ratios of hyperbolic sines for all $N\geq 0$. In particular, its dependence on $q$ is carried only by the factor $\log(q)$ appearing inside the hyperbolic sines, which motivates the following lemma:     

\begin{lemma}[The Quadratic Recurrence]\label{quad_rec}
Set $F_0(t,\beta)=F_1(t,\beta)=1$ for all $\beta\in\C$ and all $t\in\R$. For $N\geq 2$, $\re(\beta)>-2/N$, and $t\in\R$, define $F_N(t,\beta)$ by the recurrence $$F_N(t,\beta):=\begin{cases}\displaystyle{\sum_{k=1}^{N-1}\frac{k}{N}\cdot\frac{\sinh\left(\frac{t}{2}\left[\left(N+\binom{N}{2}\beta\right)\left(1-\frac{2k}{N}\right)+1\right]\right)}{\sinh\left(\frac{t}{2}\left[\left(N+\binom{N}{2}\beta\right)-1\right]\right)}\cdot F_{N-k}(t,\beta)\cdot F_k(t,\beta)}&\text{if $t\neq 0$,}\\ \\\displaystyle{\sum_{k=1}^{N-1}\frac{k}{N}\cdot\frac{\left(N+\binom{N}{2}\beta\right)\left(1-\frac{2k}{N}\right)+1}{\left(N+\binom{N}{2}\beta\right)-1}\cdot F_{N-k}(0,\beta)\cdot F_k(0,\beta)}&\text{if $t=0$.}\end{cases}$$
\begin{enumerate}
\item[(a)] For fixed $N\geq 2$ and fixed $t$, the function $\beta\mapsto F_N(t,\beta)$ is holomorphic for $\re(\beta)>-2/N$.
\item[(b)] For fixed $N\geq 2$ and fixed $\beta$, the function $t\mapsto F_N(t,\beta)$ is defined, smooth, and even on $\R$.
\end{enumerate}
\end{lemma}
Both parts of the Quadratic Recurrence are straightforward to verify by induction. An interesting and immediate consequence of The Quadratic Recurrence and the preceding discussion is the formula $$Z_N(R,\beta)=N!q^{\frac{1}{2}\binom{N}{2}\beta}F_N(\log(q),\beta)~.$$ It offers a computationally efficient alternative to the ``univariate case" of  \Cref{R_prop} (when $s_{ij}=\beta$ for all $i<j$) and extends $Z_N(R,\beta)$ to a smooth function of $q\in(0,\infty)$. Moreover, it transforms nicely under the involution $q\mapsto q^{-1}$: $$Z_N(R,\beta)\big|_{q\mapsto q^{-1}}=N!q^{-\frac{1}{2}\binom{N}{2}\beta}F_N\left(\log(q^{-1}),\beta\right)=N!q^{-\frac{1}{2}\binom{N}{2}\beta}F_N\left(\log(q),\beta\right)=q^{-\binom{N}{2}\beta}Z_N(R,\beta)~.$$

The Quadratic Recurrence serves the projective analogue as well. Expanding \eqref{P(K)=RP} into powers of $f$ yields the coefficient equations
\begin{equation}\label{Z_N/N!}
\frac{Z_N(\P^1(K),\beta)}{N!}=\sum_{k=0}^N\left(\frac{q}{q+1}\right)^N\frac{Z_{N-k}(R,\beta)}{(N-k)!}\frac{Z_k(P,\beta)}{k!}\qquad\text{for all }N\geq 0,
\end{equation}
and the identities $Z_j(P,\beta)=q^{-j-\binom{j}{2}\beta}Z_j(R,\beta)$ and $Z_j(R,\beta)=j!q^{\frac{1}{2}\binom{j}{2}\beta}F_j(\log(q),\beta)$ allow the $k$th summand to be rewritten as $$\left(\frac{q}{q+1}\right)^N\frac{Z_{N-k}(R,\beta)}{(N-k)!}\frac{Z_k(P,\beta)}{k!}=\frac{q^{\frac{1}{2}\left(N+\binom{N}{2}\beta\right)\left(1-\frac{2k}{N}\right)}}{\left(2\cosh\left(\frac{\log(q)}{2}\right)\right)^N}\cdot F_{N-k}(\log(q),\beta)\cdot F_k(\log(q),\beta)\qquad\text{for $N\geq 1$}.$$ Thus, adding two copies of the sum in \eqref{Z_N/N!} together, pairing the $k$th term of the first copy with the $(N-k)$th term of the second copy, and dividing by $2$ gives $$\frac{Z_N(\P^1(K),\beta)}{N!}=\sum_{k=0}^N\frac{\cosh\left(\frac{\log(q)}{2}\left(N+\binom{N}{2}\beta\right)\left(1-\frac{2k}{N}\right)\right)}{\left(2\cosh\left(\frac{\log(q)}{2}\right)\right)^N}\cdot F_{N-k}(\log(q),\beta)\cdot F_k(\log(q),\beta)\qquad\text{for $N\geq 1$},$$ which is valid for $\re(\beta)>-2/N$. Through this formula, $Z_N(\P^1(K),\beta)$ clearly extends to a smooth function of $q\in(0,\infty)$ and is invariant under the involution $q\mapsto q^{-1}$. We conclude this section by summarizing these observations:

\begin{theorem}[Efficient Formulas and Functional Equations]
Suppose $N\geq 2$ and $\re(\beta)>-2/N$, and define $(F_k(t,\beta))_{k=0}^N$ as in \Cref{quad_rec}. The $N$th canonical partition functions are given by the formulas
\begin{align*}
Z_N(R,\beta)&=N!q^{\frac{1}{2}\binom{N}{2}\beta}F_N(\log(q),\beta)\qquad\text{and}\\ \\
Z_N(\P(K),\beta)&=N!\sum_{k=0}^N\frac{\cosh\left(\frac{\log(q)}{2}\left(N+\binom{N}{2}\beta\right)\left(1-\frac{2k}{N}\right)\right)}{\left(2\cosh\left(\frac{\log(q)}{2}\right)\right)^N}\cdot F_{N-k}(\log(q),\beta)\cdot F_k(\log(q),\beta)~,
\end{align*}
which extend $Z_N(R,\beta)$ and $Z_N(\P^1(K),\beta)$ to smooth functions of $q\in(0,\infty)$ satisfying $$Z_N(R,\beta)\big|_{q\mapsto q^{-1}}=q^{-\binom{N}{2}\beta}Z_N(R,\beta)\qquad\text{and}\qquad Z_N(\P^1(K),\beta)\big|_{q\mapsto q^{-1}}=Z_N(\P^1(K),\beta)~.$$
\end{theorem}     

It should be noted here that the $q\mapsto q^{-1}$ functional equation at left is a special case of the one proved in \cite{DenMeus}, and that both functional equations closely resemble the ones in \cite{Voll10}.

\newpage
\section{Proofs of the main results}\label{3}
This section will establish the proofs of \Cref{main1,main2}. The common step in both is a decomposition of $(\P^1(K))^N$ into $(q+1)^N$ cells that are isometrically isomorphic to $P^N$, which combines with the metric and measure properties in \Cref{1_2} to create the key relationship between the canonical partition functions for $X=\P^1(K)$ and $X=P$. We will prove this relationship first, then conclude the proofs of \Cref{main1,main2} in their own subsections.   

\subsection{Decomposing the integral over $(\P^1(K))^N$}
We begin with an integer $N\geq 2$ that shall remain fixed for the rest of this section, reserve the symbol $\bm{s}$ for a complex tuple $(s_{ij})_{1\leq i<j\leq N}$, and fix the following notation to better organize the forthcoming arguments:
\begin{notation}
Let $I$ be a subset of $[N]=\{1,\dots,N\}$.
\begin{itemize}
\item For any set $X$ we write $X^I$ for the product $\prod_{i\in I}X=\{x_I=(x_i)_{i\in I}:x_i\in X\}$ and assume $X^I$ has the product topology if $X$ is a topological space.
\item We write $\mu^I$ for the product Haar measure on $K^I$ satisfying $\mu^I(R^I)=1$, and we make this consistent for $I=\varnothing$ by giving the singleton space $K^\varnothing=R^\varnothing=\{0\}$ measure 1. We also write $\nu^I$ for the product measure on $(\P^1(K))^I$, with the same convention for $I=\varnothing$.
\item For a measurable subset $X\subset K$ we set $\mathcal{Z}_\varnothing(X,\bm{s}):=1$ and $$\mathcal{Z}_I(X,\bm{s}):=\int_{X^I}\prod_{\substack{i<j\\i,j\in I}}|x_i-x_j|^{s_{ij}}\,d\mu^I\qquad\text{if}\quad I\neq\varnothing~.$$ Note that $\mathcal{Z}_I(X,\bm{s})$ is constant with respect to the entry $s_{ij}$ if $i\in[N]\setminus I$ or $j\in[N]\setminus I$, and it is equal to $\mathcal{Z}_N(X,\bm{s})$ if $I=[N]$.
\item Using the constant $q=\#(R/P)$, we write $(I_0,\dots,I_q)\vdash[N]$ for an \emph{ordered} partition of $[N]$ into at most $q+1$ parts. That is, $(I_0,\dots,I_q)\vdash[N]$ means $I_0,\dots,I_q$ are $q+1$ disjoint ordered subsets of $[N]$ with union equal to $[N]$, where some $I_k$ may be empty.
\end{itemize}
\end{notation}

In addition to the above, it will be useful to consider $I$-analogues of splitting chains:

\begin{definition}\label{Isplchdef}
Suppose $I\subset[N]$. An \emph{$I$-splitting chain} of length $L\geq 0$ is a tuple $\spl=(\ptn_0,\dots,\ptn_L)$ of partitions of $I$ satisfying $$\{I\}=\ptn_0>\ptn_1>\ptn_2>\dots>\ptn_L=\{\{i\}:i\in I\}~.$$ If $\#I\geq 2$, we define $\mathcal{B}(\spl)$, $\ell_\spl(\lambda)$, and $\deg_\spl(\lambda)\in\{2,3,\dots,\#I\}$ just as in \Cref{splchdef}. Otherwise $\mathcal{B}(\spl)$ will be treated as the empty set and there is no need to define $\ell_\spl$ or $\deg_\spl$. Finally, we call an $I$-splitting chain $\spl$ \emph{reduced} if each $\lambda\in\mathcal{B}(\spl)$ satisfies $\lambda\in\ptn_\ell\iff\ell=\ell_\spl(\lambda)$, write $\mathcal{R}_I$ for the set of reduced $I$-splitting chains, and define $$\Omega_I:=\bigcap_{\substack{\lambda\subset I\\\#\lambda>1}}\left\{\bm{s}:\re(e_\lambda(\bm{s}))>0\right\}\qquad\text{where}\qquad e_\lambda(\bm{s}):=\#\lambda-1+\sum_{\substack{i<j\\i,j\in\lambda}}s_{ij}~.$$ 
\end{definition}

Note that $\mathcal{R}_\varnothing=\varnothing$ because $I=\varnothing$ has no partitions, $\Omega_\varnothing=\C^{N(N-1)/2}$ because $\Omega_\varnothing$ is an intersection of subsets of $\C^{N(N-1)/2}$ taken over an empty index set, and $e_\varnothing(\bm{s})=-1$ for a similar reason. For each singleton $\{i\}$, the set $\mathcal{R}_{\{i\}}$ is comprised of a single splitting chain of length zero, $\Omega_{\{i\}}=\C^{N(N-1)/2}$ for the same reason as the $I=\varnothing$ case, and similarly $e_{\{i\}}(\bm{s})=0$. At the other extreme, taking $I=[N]$ in \Cref{Isplchdef} recovers \Cref{splchdef} and gives $\Omega_I=\Omega_N$.

\begin{proposition}\label{Iold}
For any $v\in\Z$ and any nonempty subset $I\subset[N]$, the integral $\mathcal{Z}_I(\pi^vR,\bm{s})$ converges absolutely if and only if $\bm{s}\in\Omega_I$, and in this case
$$\mathcal{Z}_I(\pi^vR,\bm{s})=\frac{1}{q^{(v-1)(e_I(\bm{s})+1)+\#I}}\sum_{\spl\in\mathcal{R}_I}\prod_{\lambda\in\mathcal{B}(\spl)}\frac{(q-1)_{\deg_{\spl}(\lambda)-1}}{q^{e_\lambda(\bm{s})}-1}~.$$ In particular, we recover \Cref{R_prop} by taking $I=[N]$ and $v=0$. 
\end{proposition}

\begin{proof} 
First suppose $I$ is a singleton, so that the product inside the integral $Z_I(\pi^vR,\bm{s})$ is empty and hence $$Z_I(\pi^vR,\bm{s})=\int_{(\pi^vR)^I}\,d\mu^I=\int_{\pi^vR}\,d\mu=q^{-v}~.$$ This integral is constant, and hence absolutely convergent, for all $\bm{s}\in\C^{N(N-1)/2}=\Omega_I$. On the other hand, $\mathcal{R}_I$ consists of a single $I$-splitting chain, namely the one-tuple $\spl=(\{I\})$. Then $\mathcal{B}(\spl)=\varnothing$ and $e_I(\bm{s})=0$ imply $$\frac{1}{q^{(v-1)(e_I(\bm{s})+1)+\#I}}\sum_{\spl\in\mathcal{R}_I}\prod_{\lambda\in\mathcal{B}(\spl)}\frac{(q-1)_{\deg_{\spl}(\lambda)-1}}{q^{e_\lambda(\bm{s})}-1}=\frac{1}{q^{(v-1)\cdot1+1}}\prod_{\lambda\in\varnothing}\frac{(q-1)_{\deg_{\spl}(\lambda)-1}}{q^{e_\lambda(\bm{s})}-1}=q^{-v}$$ as well, so the claim holds for any singleton subset $I\subset[N]$. Now suppose $I$ is not a singleton. By relabeling $I$ we may assume $I=[n]$ where $2\leq n\leq N$. By Proposition 3.15 and Lemma 3.16(c) in \cite{Web20}, the integral $$\mathcal{Z}_I(R,\bm{s})=\mathcal{Z}_n(R,\bm{s})=\int_{R^n}\prod_{1\leq i<j\leq n}|x_i-x_j|^{s_{ij}}\,d\mu^N$$ converges absolutely if and only if $\bm{s}$ belongs to the intersection
\begin{equation}\label{polytope}
\bigcap_{\spl\in\mathcal{R}_n}\bigcap_{\lambda\in\mathcal{B}(\spl)}\left\{\bm{s}:\re(e_{\lambda}(\bm{s}))>0\right\}~,
\end{equation}
and for such $\bm{s}$ we have $$\mathcal{Z}_n(R,\bm{s})=q^{e_{[n]}(\bm{s})+1-n}\sum_{\spl\in\mathcal{R}_n}\prod_{\lambda\in\mathcal{B}(\spl)}\frac{(q-1)_{\deg_{\spl}(\lambda)-1}}{q^{e_\lambda(\bm{s})}-1}~.$$ Changing variables in the integral $\mathcal{Z}_n(R,\bm{s})$ by the homothety $R^n\to(\pi^vR)^n$ gives $$\mathcal{Z}_N(\pi^vR,\bm{s})=q^{-v(e_{[N]}(\bm{s})+1)}\cdot\mathcal{Z}_n(R,\bm{s})=\frac{1}{q^{(v-1)(e_{[n]}(\bm{s})+1)+n}}\sum_{\spl\in\mathcal{R}_n}\prod_{\lambda\in\mathcal{B}(\spl)}\frac{(q-1)_{\deg_{\spl}(\lambda)-1}}{q^{e_\lambda(\bm{s})}-1}~,$$ and the first equality implies that the domain of absolute convergence for $\mathcal{Z}_n(\pi^vR,\bm{s})$ is also the intersection appearing in \eqref{polytope}. But every subset $\lambda\subset[n]$ with $\#\lambda>1$ appears as a branch in at least one reduced splitting chain of order $n$, so the intersection in \eqref{polytope} is precisely $\Omega_{[n]}$. Therefore the claim holds for $I=[n]$, and hence for any non-singleton subset $I\subset[N]$.\\ 
\end{proof}

The $v=1$ case of \Cref{Iold} has an important relationship with the main result of this section, which is the following theorem.

\begin{theorem}\label{main}
For each $N\geq 2$, the integral $\mathcal{Z}_N(\P^1(K),\bm{s})$ converges absolutely if and only if $\bm{s}\in\Omega_N$, and in this case $$\mathcal{Z}_N(\P^1(K),\bm{s})=\left(\frac{q}{q+1}\right)^N\sum_{(I_0,\dots,I_q)\vdash[N]}\prod_{k=0}^q\mathcal{Z}_{I_k}(P,\bm{s})~.$$
\end{theorem}

\begin{proof}
The partition of $\P^1(K)$ in \eqref{P(K)_decomp} can be rewritten in the form $$\P^1(K)=\bigsqcup_{k=0}^q\phi_k(B_1[0:1])~,$$ where $\phi_k\in PGL_2(R)$ is the element represented by $\left(\begin{smallmatrix}1&0\\0&1\end{smallmatrix}\right)$ if $k=0$, $\left(\begin{smallmatrix}1&\xi^{k-1}\\0&1\end{smallmatrix}\right)$ if $0<k<q$, or $\left(\begin{smallmatrix}0&1\\1&0\end{smallmatrix}\right)$ if $k=q$. This leads to a partition of the $N$-fold product, $$(\P^1(K))^N=\bigsqcup_{(I_0,\dots,I_q)\vdash[N]}C(I_0,\dots,I_q)~,$$ where each part is a ``cell" of the form
\begin{align*}
C(I_0,\dots,I_q)&:=\{([x_{1,0}:x_{1,1}],\dots,[x_{N,0}:x_{N,1}])\in(\P^1(K))^N:[x_{i,0}:x_{i,1}]\in\phi_k(B_1[0:1])\iff i\in I_k\}\\
&=\prod_{k=0}^q(\phi_k(B_1[0:1]))^{I_k}~.
\end{align*}
Accordingly, the integral $\mathcal{Z}_N(\P^1(K),\bm{s})$ breaks into a sum of integrals of the form
\begin{equation}\label{cube_integral}
\int_{C(I_0,\dots,I_q)}\prod_{1\leq i<j\leq N}\delta([x_{i,0}:x_{i,1}],[x_{j,0}:x_{j,1}])^{s_{ij}}\,d\nu^N~,
\end{equation}
summed over all $(I_0,\dots,I_q)\vdash[N]$. Since every cell $C(I_0,\dots,I_q)$ has positive measure, the integral $\mathcal{Z}_N(\P^1(K),\bm{s})$ converges absolutely if and only if the integral in \eqref{cube_integral} converges absolutely for all $(I_0,\dots,I_q)\vdash[N]$. Fix one $(I_0,\dots,I_q)$ for the moment. By \eqref{P(K)_decomp} and the definition of the $\phi_k$'s above, note that the entries of each tuple $([x_{1,0}:x_{1,1}],\dots,[x_{N,0}:x_{N,1}])\in C(I_0,\dots,I_q)$ satisfy $\delta([x_{i,0}:x_{i,1}],[x_{j,0}:x_{j,1}])^{s_{ij}}=1$ if and only if $i$ and $j$ belong to different parts of $(I_0,\dots,I_q)$. Therefore the integrand in \eqref{cube_integral} factors as
\begin{align*}
\prod_{1\leq i<j\leq N}\delta([x_{i,0}:x_{i,1}],[x_{j,0}:x_{j,1}])^{s_{ij}}&=\prod_{k=0}^q\prod_{\substack{i<j\\i,j\in I_k}}\delta([x_{i,0}:x_{i,1}],[x_{j,0}:x_{j,1}])^{s_{ij}}~,
\end{align*}
and the measure on $C(I_0,\dots,I_q)$ factors in a similar way, namely $\prod_{k=0}^qd\nu^{I_k}$ where $\nu^{I_k}$ is the product measure on $(\P^1(K))^{I_k}$. Now Fubini's Theorem for positive functions and $PGL_2(R)$-invariance give
\begin{align*}
\int_{C(I_0,\dots,I_q)}&\Bigg|\prod_{1\leq i<j\leq N}\delta([x_{i,0}:x_{i,1}],[x_{j,0}:x_{j,1}])^{s_{ij}}\Bigg|_\C\,d\nu^N\\
&=\prod_{k=0}^q\int_{(\phi_k(B_1[0:1]))^{I_k}}\Bigg|\prod_{\substack{i<j\\i,j\in I_k}}\delta([x_{i,0}:x_{i,1}],[x_{j,0}:x_{j,1}])^{s_{ij}}\Bigg|_\C\,d\nu^{I_k}\\
&=\prod_{k=0}^q\int_{(B_1[0:1])^{I_k}}\Bigg|\prod_{\substack{i<j\\i,j\in I_k}}\delta([x_{i,0}:x_{i,1}],[x_{j,0}:x_{j,1}])^{s_{ij}}\Bigg|_\C\,d\nu^{I_k}~,
\end{align*}
so the integral in \eqref{cube_integral} converges absolutely if and only if all $q+1$ of the integrals of the form
\begin{equation}\label{I_k_integral}
\int_{(B_1[0:1])^{I_k}}\prod_{\substack{i<j\\i,j\in I_k}}\delta([x_{i,0}:x_{i,1}],[x_{j,0}:x_{j,1}])^{s_{ij}}\,d\nu^{I_k}
\end{equation}
converge absolutely. The change of variables $P^{I_k}\to(B_1[0:1])^{I_k}$ given by $\iota:P\to B_1[0:1]$ in each coordinate, along with \eqref{metric_rule}, \eqref{ball_rule}, and \eqref{measure_rule}, allows the integral in \eqref{I_k_integral} to be rewritten as $(\frac{q}{q+1})^{\#I_k}\mathcal{Z}_{I_k}(P,\bm{s})$, and thus \Cref{Iold} implies that it converges absolutely if and only if $\bm{s}\in\Omega_{I_k}$. Therefore the integral over $C(I_0,\dots,I_q)$ in \eqref{cube_integral} converges absolutely if and only if $\bm{s}\in\Omega_{I_0}\cap\dots\cap\Omega_{I_q}$, and in this case Fubini's Theorem for absolutely integrable functions, $PGL_2(R)$-invariance, and the change of variables above allow it to be rewritten as
\begin{align*}
\int_{C(I_0,\dots,I_q)}&\prod_{1\leq i<j\leq N}\delta([x_{i,0}:x_{i,1}],[x_{j,0}:x_{j,1}])^{s_{ij}}\,d\nu^N\\
&=\prod_{k=0}^q\int_{(B_1[0:1])^{I_k}}\prod_{\substack{i<j\\i,j\in I_k}}\delta([x_{i,0}:x_{i,1}],[x_{j,0}:x_{j,1}])^{s_{ij}}\,d\nu^{I_k}\\
&=\prod_{k=0}^q\left(\frac{q}{q+1}\right)^{\#I_k}\mathcal{Z}_{I_k}(P,\bm{s})\\
&=\left(\frac{q}{q+1}\right)^N\prod_{k=0}^q\mathcal{Z}_{I_k}(P,\bm{s})~.
\end{align*}

Finally, since $\mathcal{Z}_N(\P^1(K),\bm{s})$ is the sum of these integrals over all $(I_0,\dots,I_q)\vdash[N]$, it converges absolutely if and only if $$\bm{s}\in\bigcap_{(I_1,\dots,I_q)\vdash[N]}\left(\Omega_{I_1}\cap\dots\cap\Omega_{I_q}\right)=\bigcap_{\substack{I\subset[N]\\\#I>1}}\Omega_I~.$$ The last equality of intersections holds because each subset $I\subset[N]$ with $\#I>1$ appears as a part in at least one of the ordered partitions $(I_1,\dots,I_q)\vdash[N]$, and because none of the parts with $\#I_k\leq 1$ affect the intersection (because $\Omega_{I_k}=\C^{N(N-1)/2}$ for such $I_k$). The intersection of $\Omega_I$ over all $I\subset[N]$ with $\#I>1$ is clearly equal to $\Omega_{[N]}=\Omega_N$ by \Cref{Isplchdef}, so the proof is complete.\\ 
\end{proof}

\subsection{Finishing the proof of Theorem 2.1}
\Cref{main} established that the integral $\mathcal{Z}_N(\P^1(K),\bm{s})$ converges absolutely if and only if $\bm{s}\in\Omega_N$, and for such $\bm{s}$ it gave
\begin{equation}\label{sum1}
\mathcal{Z}_N(\P^1(K),\bm{s})=\left(\frac{q}{q+1}\right)^N\sum_{(I_0,\dots,I_q)\vdash[N]}\prod_{k=0}^q\mathcal{Z}_{I_k}(P,\bm{s})~.
\end{equation}
It remains to show that the righthand sum can be converted into the sum over $\spl\in\mathcal{R}_N$ proposed in \Cref{main1}. 

\begin{proof}[Proof of \Cref{main1}] We begin by breaking the terms of the sum in \eqref{sum1} into two main groups. The simpler group is indexed by those $(I_0,\dots,I_q)$ with $I_j=[N]$ for some $j$ and $I_k=\varnothing$ for all $k\neq j$, in which case $\mathcal{Z}_{I_j}(P,\bm{s})=\mathcal{Z}_N(P,\bm{s})$ and $\mathcal{Z}_{I_k}(P,\bm{s})=1$ for all $k\neq j$. Therefore each of the group's $q+1$ terms (one for each $j\in\{0,\dots,q\}$) contributes the quantity $\prod_{k=0}^q\mathcal{Z}_{I_k}(P,\bm{s})=\mathcal{Z}_N(P,\bm{s})$ to the sum in \eqref{sum1} for a total contribution of
\begin{equation}\label{term1}
(q+1)\mathcal{Z}_N(P,\bm{s})=\frac{q+1}{q^N}\sum_{\spl\in\mathcal{R}_N}\prod_{\lambda\in\mathcal{B}(\spl)}\frac{(q-1)_{\deg_\spl(\lambda)-1}}{q^{e_\lambda(\bm{s})}-1}
\end{equation}
by the $v=1$ and $I=[N]$ case of \Cref{Iold}. The other group of terms is indexed by the ordered partitions $(I_0,\dots,I_q)\vdash[N]$ satisfying $I_0,\dots,I_q\subsetneq[N]$. To deal with them carefully, we fix one such $(I_0,\dots,I_q)$ for the moment, and note that the number $d$ of nonempty parts $I_k$ must be at least 2. Thus we have indices $k_1,\dots,k_d\in\{0,\dots,q\}$ with $I_{k_j}\neq\varnothing$, and for every $k\in\{0,\dots,q\}\setminus\{k_1,\dots,k_d\}$ we have $I_k=\varnothing$ and hence $\mathcal{Z}_{I_k}(P,\bm{s})=1$. For the nonempty sets $I_{k_j}$, \Cref{Iold} expands $\mathcal{Z}_{I_{k_j}}(P,\bm{s})$ as a sum over $\mathcal{R}_{I_{k_j}}$ (whose elements shall be denoted $\spl_j$ instead of $\spl$) and hence
\begin{align*}
\prod_{k=0}^q\mathcal{Z}_{I_k}(P,\bm{s})&=\prod_{j=1}^d\frac{1}{q^{\#I_{k_j}}}\sum_{\spl_j\in\mathcal{R}_{I_{k_j}}}\prod_{\lambda\in\mathcal{B}(\spl_j)}\frac{(q-1)_{\deg_{\spl_j}(\lambda)-1}}{q^{e_\lambda(\bm{s})}-1}\\
&=\frac{1}{q^N}\sum_{(\spl_1,\dots,\spl_d)\in\mathcal{R}_{I_{k_1}}\times\cdots\times\mathcal{R}_{I_{k_d}}}\prod_{j=1}^d\prod_{\lambda\in\mathcal{B}(\spl_j)}\frac{(q-1)_{\deg_{\spl_j}(\lambda)-1}}{q^{e_\lambda(\bm{s})}-1}\\
&=\frac{1}{q^N}\sum_{(\spl_1,\dots,\spl_d)\in\mathcal{R}_{I_{k_1}}\times\cdots\times\mathcal{R}_{I_{k_d}}}\prod_{\lambda\in\mathcal{B}(\spl_1)\sqcup\cdots\sqcup\mathcal{B}(\spl_d)}\frac{(q-1)_{\deg_{\spl_j}(\lambda)-1}}{q^{e_\lambda(\bm{s})}-1}~.
\end{align*}
We now make use of a simple correspondence between the tuples $(\spl_1,\dots,\spl_d)\in\mathcal{R}_{I_{k_1}}\times\cdots\times\mathcal{R}_{I_{k_d}}$ and the reduced splitting chains $\spl=(\ptn_0,\ptn_1,\dots,\ptn_L)\in\mathcal{R}_N$ satisfying $\ptn_1=\{I_{k_1},\dots,I_{k_d}\}$. To establish it, note that each $\spl\in\mathcal{R}_N$ corresponds uniquely to its branch set $\mathcal{B}(\spl)$ (Lemma 2.5(b) of \cite{Web20}), which generalizes in an obvious way to reduced $I$-splitting chains (for any nonempty $I\subset[N]$). Now if $\spl=(\ptn_0,\ptn_1,\dots,\ptn_L)\in\mathcal{R}_N$ satisfies $\ptn_1=\{I_{k_1},\dots,I_{k_d}\}$, the corresponding branch set $\mathcal{B}(\spl)$ decomposes as $$\mathcal{B}(\spl)=\{[N]\}\sqcup\bigsqcup_{j=1}^d\{\lambda\in\mathcal{B}(\spl):\lambda\subset I_{k_j}\}~.$$ Each of the sets $\{\lambda\in\mathcal{B}(\spl):\lambda\subset I_{k_j}\}$ is the branch set $\mathcal{B}(\spl_j)$ for a unique $\spl_j\in\mathcal{R}_{I_{k_j}}$, so in this sense $\spl$ ``breaks" into a unique tuple $(\spl_1,\dots,\spl_d)\in\mathcal{R}_{I_{k_1}}\times\cdots\times\mathcal{R}_{I_{k_d}}$. On the other hand, any tuple $(\spl_1,\dots,\spl_d)\in\mathcal{R}_{I_{k_1}}\times\cdots\times\mathcal{R}_{I_{k_d}}$ can be ``assembled" as follows. Since $\{I_{k_1},\dots,I_{k_d}\}$ is a partition of $[N]$, taking the union of the $d$ branch sets $\mathcal{B}(\spl_1),\dots,\mathcal{B}(\spl_d)$ and the singleton $\{[N]\}$ forms the branch set $\mathcal{B}(\spl)$ for a unique $\spl\in\mathcal{R}_N$. It is clear that ``breaking" and ``assembling" are inverses, giving a correspondence $\mathcal{R}_N\longleftrightarrow\mathcal{R}_{I_{k_1}}\times\cdots\times\mathcal{R}_{I_{k_d}}$ under which each identification $\spl\longleftrightarrow(\spl_1,\dots,\spl_d)$ amounts to a branch set equation, i.e., $$\mathcal{B}(\spl)\setminus\{[N]\}=\mathcal{B}(\spl_1)\sqcup\cdots\sqcup\mathcal{B}(\spl_d)~.$$ In particular, each $\lambda\in\mathcal{B}(\spl)\setminus\{[N]\}$ is contained in exactly one $\mathcal{B}(\spl_j)$, and $\deg_\spl(\lambda)=\deg_{\spl_j}(\lambda)$ by \Cref{splchdef} in this case. These facts allow the sum over $\mathcal{R}_{I_1}\times\cdots\times\mathcal{R}_{I_d}$ above to be rewritten as a sum over all $\spl\in\mathcal{R}_N$ with $\ptn_1=\{I_{k_1},\dots,I_{k_q}\}$, and each product over $\lambda\in\mathcal{B}(\spl_{k_1})\sqcup\cdots\sqcup\mathcal{B}(\spl_{k_d})$ inside it is simply a product over $\lambda\in\mathcal{B}(\spl)\setminus\{[N]\}$. We conclude that an ordered partition $(I_0,\dots,I_q)\vdash[N]$ with $I_0,\dots,I_q\subsetneq[N]$ contributes the quantity 
\begin{equation}\label{term2}
\prod_{k=0}^q\mathcal{Z}_{I_k}(P,\bm{s})=\frac{1}{q^N}\sum_{\substack{\spl\in\mathcal{R}_N\\\ptn_1=\{I_{k_1},\dots,I_{k_d}\}}}\prod_{\lambda\in\mathcal{B}(\spl)\setminus\{[N]\}}\frac{(q-1)_{\deg_\spl(\lambda)-1}}{q^{e_\lambda(\bm{s})}-1}
\end{equation}
to the sum in \eqref{sum1}, where $\{I_{k_1},\dots,I_{k_d}\}$ is the (unordered) subset of nonempty parts in that particular ordered partition. We must now total the contribution in \eqref{term2} over all possible $(I_0,\dots,I_q)\vdash[N]$ with $I_0,\dots,I_q\subsetneq[N]$. Given a partition $\{\lambda_1,\dots,\lambda_d\}\vdash[N]$ with $d\geq 2$, note that there are precisely $(q+1)_d=(q+1)\cdot(q)_{d-1}$ ordered partitions $(I_0,\dots,I_q)\vdash[N]$ such that $\{I_{k_1},\dots,I_{k_d}\}=\{\lambda_1,\dots,\lambda_d\}$. Therefore summing \eqref{term2} over all $(I_0,\dots,I_q)\vdash[N]$ with $I_0,\dots,I_q\subsetneq[N]$ gives
\begin{align*}
\sum_{\substack{(I_0,\dots,I_q)\vdash[N]\\I_0,\dots,I_q\subsetneq[N]}}\prod_{k=0}^q\mathcal{Z}_{I_k}(P,\bm{s})&=\frac{1}{q^N}\sum_{\substack{(I_0,\dots,I_q)\vdash[N]\\I_0,\dots,I_q\subsetneq[N]}}\sum_{\substack{\spl\in\mathcal{R}_N\\\ptn_1=\{I_{k_1},\dots,I_{k_d}\}}}\prod_{\lambda\in\mathcal{B}(\spl)\setminus\{[N]\}}\frac{(q-1)_{\deg_\spl(\lambda)-1}}{q^{e_\lambda(\bm{s})}-1}\\
&=\frac{q+1}{q^N}\sum_{\substack{\{\lambda_1,\dots,\lambda_d\}\vdash[N]\\d\geq 2}}(q)_{d-1}\sum_{\substack{\spl\in\mathcal{R}_N\\\ptn_1=\{\lambda_1,\dots,\lambda_d\}}}\prod_{\lambda\in\mathcal{B}(\spl)\setminus\{[N]\}}\frac{(q-1)_{\deg_\spl(\lambda)-1}}{q^{e_\lambda(\bm{s})}-1}~.
\end{align*}
Given a partition $\{\lambda_1,\dots,\lambda_d\}\vdash[N]$, those splitting chains $\spl\in\mathcal{R}_N$ with $\ptn_1=\{\lambda_1,\dots,\lambda_d\}$ all have $\deg_\spl([N])=\#\ptn_1=d$ by \Cref{splchdef}. Moreover, no $\spl\in\mathcal{R}_N$ is missed or repeated in the sum of sums above, so it can be rewritten as
\begin{align*}
\sum_{\substack{(I_0,\dots,I_q)\vdash[N]\\I_0,\dots,I_q\subsetneq[N]}}\prod_{k=0}^q\mathcal{Z}_{I_k}(P,\bm{s})&=\frac{q+1}{q^N}\sum_{\spl\in\mathcal{R}_N}(q)_{\deg_\spl([N])-1}\prod_{\lambda\in\mathcal{B}(\spl)\setminus\{[N]\}}\frac{(q-1)_{\deg_\spl(\lambda)-1}}{q^{e_\lambda(\bm{s})}-1}\\
&=\frac{q+1}{q^N}\sum_{\spl\in\mathcal{R}_N}\frac{(q)_{\deg_\spl([N])-1}}{(q-1)_{\deg_\spl([N])-1}}\cdot(q^{e_{[N]}(\bm{s})}-1)\prod_{\lambda\in\mathcal{B}(\spl)}\frac{(q-1)_{\deg_\spl(\lambda)-1}}{q^{e_\lambda(\bm{s})}-1}\\
&=\frac{q+1}{q^N}\sum_{\spl\in\mathcal{R}_N}\frac{q^{N+\sum_{i<j}s_{ij}}-q}{q+1-\deg_\spl([N])}\prod_{\lambda\in\mathcal{B}(\spl)}\frac{(q-1)_{\deg_\spl(\lambda)-1}}{q^{e_\lambda(\bm{s})}-1}~.
\end{align*}
Note that the summand for each $\spl\in\mathcal{R}_N$ is still defined for any prime power $q$ since the denominators $(q-1)_{\deg_\spl([N])-1}$ and $q+1-\deg_\spl([N])$ (which vanish when $q=\deg_\spl([N])-1$) are cancelled by the numerator $(q-1)_{\deg_\spl([N])-1}$ appearing in the product over $\lambda\in\mathcal{B}(\spl)$. Finally, we evaluate the righthand side of \eqref{sum1} by combining the sum directly above with that in \eqref{term1} and multiplying through by $(\frac{q}{q+1})^N$. This yields the desired formula for $\mathcal{Z}_N(\P^1(K),\bm{s})$: 
\begin{align*}
\mathcal{Z}_N(\P^1(K),\bm{s})&=\frac{1}{(q+1)^{N-1}}\sum_{\spl\in\mathcal{R}_N}\left(1+\frac{q^{N+\sum_{i<j}s_{ij}}-q}{q+1-\deg_\spl([N])}\right)\prod_{\lambda\in\mathcal{B}(\spl)}\frac{(q-1)_{\deg_\spl(\lambda)-1}}{q^{e_\lambda(\bm{s})}-1}\\
&=\frac{1}{(q+1)^{N-1}}\sum_{\spl\in\mathcal{R}_N}\frac{q^{N+\sum_{i<j}s_{ij}}+1-\deg_\spl([N])}{q+1-\deg_\spl([N])}\prod_{\lambda\in\mathcal{B}(\spl)}\frac{(q-1)_{\deg_\spl(\lambda)-1}}{q^{e_\lambda(\bm{s})}-1}
\end{align*}
\end{proof}

\subsection{Finishing the proof of Theorem 2.3}

Our final task is to prove the $(q+1)$th Power Law, which we noted in \Cref{2_2} is equivalent to the equations in \eqref{P(K)coeff}. That is, it remains to prove $$\frac{Z_N(\P^1(K),\beta)}{N!}=\sum_{\substack{N_0+\dots+N_q=N\\N_0,\dots,N_q\geq 0}}\prod_{k=0}^q\left(\frac{q}{q+1}\right)^{N_k}\frac{Z_{N_k}(P,\beta)}{N_k!}\qquad\text{for all $\beta>0$ and $N\geq 0$}.$$

\begin{proof}
Fix $N\geq 0$ and $\beta>0$, and fix $\bm{s}$ via $s_{ij}=\beta$ for all $i<j$, so that $\mathcal{Z}_N(\P^1(K),\bm{s})=Z_N(\P^1(K),\beta)$ and $\mathcal{Z}_I(P,\bm{s})=Z_{\#I}(P,\beta)$ for any subset $I\subset[N]$. The formula in \Cref{main} relates these functions of $\beta$ via 
\begin{align*}
Z_N(\P^1(K),\beta)&=\mathcal{Z}_N(\P^1(K),\bm{s})\\
&=\left(\frac{q}{q+1}\right)^N\sum_{(I_0,\dots,I_q)\vdash[N]}\prod_{k=0}^q\mathcal{Z}_{\#I_k}(P,\bm{s})\\
&=\sum_{(I_0,\dots,I_q)\vdash[N]}\prod_{k=0}^q\left(\frac{q}{q+1}\right)^{\#I_k}Z_{\#I_k}(P,\beta)~.
\end{align*}
For each choice of $q+1$ ordered integers $N_0,\dots,N_q\geq 0$ satisfying $N_0+\dots+N_q=N$, there are precisely $$\binom{N}{N_0,\dots,N_q}=\frac{N!}{N_0!\cdots N_q!}$$ ordered partitions $(I_0,\dots,I_q)\vdash[N]$ satisfying $\#I_0=N_0,\dots,\#I_q=N_q$. Finally, grouping ordered partitions according to all possible ordered integer choices establishes the desired equation:  
\begin{align*}
\frac{Z_N(\P^1(K),\bm{s})}{N!}&=\frac{1}{N!}\cdot\sum_{(I_0,\dots,I_q)\vdash[N]}\prod_{k=0}^q\left(\frac{q}{q+1}\right)^{\#I_k}Z_{\#I_k}(P,\beta)\\
&=\frac{1}{N!}\cdot\sum_{\substack{N_0+\dots+N_q=N\\N_0,\dots,N_q\geq 0}}\binom{N}{N_0,\dots,N_q}\prod_{k=0}^q\left(\frac{q}{q+1}\right)^{N_k}Z_{N_k}(P,\beta)\\
&=\sum_{\substack{N_0+\dots+N_q=N\\N_0,\dots,N_q\geq 0}}\prod_{k=0}^q\left(\frac{q}{q+1}\right)^{N_k}\frac{Z_{N_k}(P,\beta)}{N_k!}~.
\end{align*}
\end{proof} 

\noindent\textbf{Acknowledgements}: I would like to thank Clay Petsche for confirming several details about the measure and metric on $\P^1(K)$, and I would like to thank Chris Sinclair for many useful suggestions regarding the Power Laws and the Quadratic Recurrence.  


\bibliography{references}
\bibliographystyle{amsalpha}

\begin{center}
\noindent\rule{4cm}{.5pt}
\vspace{.25cm}

\noindent {\sc \small Joe Webster}\\
{\small Department of Mathematics, University of Virginia, Charlottesville, VA 22903} \\
email: {\tt nsr6sf@virginia.edu}
\end{center}

\end{document}